\documentclass[10pt,twoside, final]{amsart}
\usepackage{amsmath,amsthm,amscd,amsfonts,amssymb,enumerate}
\usepackage{graphicx}
\usepackage{color}
\usepackage[colorlinks]{hyperref}

\newtheorem{theorem}{Theorem}[section]
\newtheorem{proposition}[theorem]{Proposition}
\newtheorem{lemma}[theorem]{Lemma}

\theoremstyle{definition}

\newtheorem{example}[theorem]{Example}
\theoremstyle{remark}
\newtheorem{remark}[theorem]{Remark}
\numberwithin{equation}{section}
\newtheorem{problem}[theorem]{Problem}

\begin{document}

\title[A note on the Variety of Secant Loci 
]{
A note on the Variety of Secant Loci  }

\author[Ali Bajravani]{Ali Bajravani}
\address[]{Department of Mathematics, Faculty of Basic Sciences, Azarbaijan Shahid Madani University, Tabriz, I. R. Iran.\\
P. O. Box: 53751-71379.}

\email{bajravani@azaruniv.ac.ir}

\begin{abstract}
We determine non-hyper elliptic curves $C$ with $g(C)\geq 9$, such that for some very ample line bundle $L$ on them and for some integers $d$ and $r$ with $0< 2r < d\leq h^0(L)+r-4$,
the dimension of the Secant Loci, $\dim V^{d-r}_{d}(L)$, attains one less than its maximum value. Then we proceed to prove that for positive integers $\gamma, d$ with 
some circumstances on $\gamma, d$ and $h^0(L)$,
 if one had $\dim V^{d-1}_{d}(L)=d-1-\gamma$ then $V^{\gamma+2}_{\gamma +3}(L)$ would be $2$-dimensional.\\
% This has been proved by M. Coppens in the special case $L=K_{C}$ with some restrictions on $g(C)$.\\



\end{abstract}
\maketitle

\vspace{-.5cm}

\section{\bf Introduction}

Assume that  $r$, $d$ and $g$ are integers with $0< 2r<d \leq g-1$. On a smooth projective algebraic curve $C$ with genus $g$, the dimension of the scheme $C^{r}_{d}$, can't exceed $d-r$.
Through Martens, Mumford and Keem theorems it is known that; proximity of $\dim C^{r}_{d}$ to $d-r$, for some $r, d$; imposes specific geometry on $C$. Based on Keem theorem, in the occurrence of 
$\dim C^{r}_{d}=d-r-2$ for some $r$ and $d$; the curve $C$ would be a three sheeted or a $4$-sheeted covering of the projective line.

It is unknown whether if one can derive these kind of geometric information for $C$ in the case that; the real dimension of the schemes of Secant Loci associated to an arbitrary very ample line bundle is close to its maximum value.  We studied this problem for curves of genus $g\geq 4$ in \cite{A. B}. After proving Martens theorem for secant loci associated to very ample line bundles on curves of genus $g\geq 4$, we established a Mumford type theorem for curves of genus $g\geq 9$. See \cite[Theorem 4.6]{A. B}. We will complete the 
next step in this direction, in section \ref{section3}.  Namely, we prove the analogue of Keem's theorem for secant loci associated to very ample line bundles on curves of genus $g\geq 9$. See Theorem \ref{theorem1}.

 Marc Coppens went far beyound this theory in \cite{Coppens}, by systematizing Martens, Mumford and Keem theorems.
Under some restrictions on the genus of $C$, he proved that the equality 
$\dim W^{1}_{\gamma+3}=1$ is guaranteed by an equality $\dim W^{1}_{d}=d-2-\gamma$ for an integer $\gamma$ and some integer $d$ with  $0\leq \gamma+3 \leq d \leq g-1-\gamma$. See problem \ref{problem}. In theorem \ref{theorem2} we prove Coppen's result for secant loci when the canonical line bundle of $C$ is replaced by an arbitrary very ample line bundle. 

Coppen's method is based on a delicate analysis of a specific irreducible component of $W^{1}_{d}$. In the absence of a suitable residuation process in our full generality situation, Coppens method seems mostly unapplicable for secant loci of arbitrary line bundles. We take a different approach. Our method relies on an inductive approach together with entering another suitably choosen line bundle in to the argument.
Probably the most unexpected advantage of our method is to remove the restrictions, imposed by Coppens, on the genus of $C$ when $\gamma \geq 3$.

In theorem \ref{theorem3}, we report a dimension computation for secant loci when some specific secant loci are empty. The emptiness assumption is hold for the canonical line bundle of a general curve, so we re-obtain the classical Brill-Noether dimension theorem for $C^1_d$'s on general curves.

In order to establish thoerems \ref{theorem2} and \ref{theorem3} we essentially need an analogue of Fulton-Harris-Lazarsfeld result on excess dimension of linear series, for secant loci. See \cite{F-H-L}.
Such an instrument has been produced only recently by M. Aprodu and E. Sernesi in \cite{A-S2}.

Since a divisor $D\in C^{1}_{\gamma +3}$ gives a $g^{1}_{\gamma +3}$ on $C$, replacing $\gamma=0, 1, 2$, Coppens result specializes to Martens, Mumford and Keem theorems, respectively. Through remark \ref{remark2}, we notify that theorem \ref{theorem1} can not be concluded from Theorem \ref{theorem2}, so it actually needs an independent proof.

\section{Notations and Backgrounds}
 Assume that $L$ is a line bundle on a smooth projective algebraic curve $C$ of genus $g$ and $d$ a positive integer. For a positive integer $k\leq d-1$ consider the subset $V_{d}^{k}(L)$ of $C_{d}$, set theoretically defined by
 $$V^{k}_{d}(L):=\lbrace D\in C_{d} \mid h^0(L(-D))\geq h^0(L)-k \rbrace.$$
 The subset $V_{d}^{k}(L)$ has a natural scheme structure.
 See \cite{A-S}, \cite{ACGH}, \cite{A. B} for more details on the scheme structure of $V_{d}^{k}(L)$ and some of its geometric properties.
 
The schemes $V^{k}_{d}(L)$ immediately generalize the well known Brill-Noether varieties $C^{r}_{d}$. As well as $C^{r}_{d}$'s the scheme of linear series on an algebraic curve, $W^{r}_{d}$'s, are of central objects in the theory of algebraic curves. C. Keem and M. Coppens have determined non-hyper elliptic curves which for them $\dim W^{r}_{d}$ attains one less than its maximum value. See \cite{Keem}, \cite{Coppens}.
%\begin{thm}[Martens]\label{Martens}
\begin{theorem}[Coppens-Keem]\label{Keem}
Let $C$ be a smooth algebraic curve of genus $g\geq 9$, and suppose that for some integers $d$ and $r$ satisfying $d\leq g+r-4, r\geq 1$
we have $\dim W^{r}_{d}=d-2r-2$. Then $C$ admits a $g^{1}_{4}$.
\end{theorem}
 In \cite{Coppens} Marc Coppens, systematizing results of Martens, Mumford and Keem; imposed problrm \ref{problem}, concerning dimensions of the varieties of linear series on $C$, see also \cite{M}:
\begin{problem}\label{problem}
%\begin{itemize}
%\item[$\bullet$]  \textit{Problem:} 
Assume that $g(C)\geq 9$ and $\gamma$ is a non-negative integer with $2\gamma +4\leq g, \quad \!\!\!\!\gamma+3\leq d\leq g-1-\gamma $. Is it true that; $ W^{1}_{\gamma+3}$ would be $1$-dimensional provided that $\dim W^{1}_{d}=d-2-\gamma$?
%\end{itemize}
\end{problem}
Once the question was treated by Martens and Mumford in the cases $\gamma=0, 1$ and by Keem in the case $\gamma=2, g(C)\geq 11$; M. Coppens affirmatively answered it for $\gamma=2, g(C)=9, 10$; 
$\gamma=3, g(C)=12, 13, 14$ and $\gamma>3, g(C)\geq (\gamma +1)(2\gamma +1)$. Meanwhile the case $\gamma=3, g(C)\geq 15$ was answered by Martens. 

We prove theorem \ref{theorem2}, where we slightly generalize and extend M. Coppens and Martens results.
 We call $C$ an exceptional curve if it is $3$-gonal, $4$-gonal, bi-elliptic or a space septic curve.

\section{Keem Theorem for Secant Loci}\label{section3}
%Assume that $C$ is a non-hyper elliptic smooth projective algebraic curve. 
In this section, we prove Coppens-Keem theorem for secant loci of very ample line bundles on non-hyper elliptic smooth projective algebraic curves of genus $g\geq 9$. 
\begin{lemma}\label{lemma1}
Let $C$ be a non-hyper elliptic curve of genus $g\geq 9$. Then $C$ is exceptional provided that
$\dim V^{2}_{3}(L)=1$, for some very ample line bundle $L$ on $C$.
\end{lemma}
\begin{proof}
% If $C$ is trigonal or a bi-elliptic curve, then $\dim V^{2}_{3}(K_{C})=1$. If $C$ is $4$-gonal, then for general $p\in C$ the line bundle $L=K(-p)$ is very ample with $\dim V^{2}_{3}(L)=1$. Consider moreover that a space septic is a $4$-gonal curve with $g(C)=9$.

This is implicitly contained in the proof of \cite[Theorem 4.6]{A. B}.
% that, if for some very ample line bundle $L$ on the curve $C$ one had $\dim V^{2}_{3}(L)=1$, then $C$ has to be exceptional
\end{proof}
\begin{example}\label{example}
For a very ample line bundle $L$ on a bi-elliptic curve $C$ and integeres $r, d$ with $0< 2r < d \leq h^0(L)-2$, we have $\dim V^{d-r}_{d}(L)=d-r-1$.  
 To see this let $\epsilon : C\rightarrow E$ be the elliptic double covering. Then moving $p$ on $ E$, the lines $<P, Q>$, where $\epsilon^{-1}(p)=\lbrace P, Q \rbrace$, sweep a cone containing $\phi_{L}(C)$ in $\mathbb{P}(H^0(L))$, the elliptic cone. Let
  $l_{i}= <P_{i}+Q_{i}>$, for $i=1,\cdots , r+1$,
  be general generating lines of the elliptic cone such that
  $<l_{1}, \cdots , l_{r+1}>=\mathbf{P}^{r+1}\subset \mathbf{P}(H^0(L))$.
Then, for $d\leq h^0(L)-2$ the divisors of type
$D=R_{1}+\cdots + R_{d-2r-2}+\sum_{i=1}^{r+1}P_{i}+Q_{i},$
where $R_{1}, \cdots , R_{d-2r-2}$ are general points on $C$, belong to $V^{d-r}_{d}(L)$; implying $\dim V^{d-r}_{d}(L)=d-r-1$.

\end{example}
\begin{proposition}\label{proposition}
Let $C$ be a non-hyper elliptic curve of genus $g\geq 4$. Assume that
$\dim V^{3}_{4}(L)=1$ for some very ample line bundle $L$ on $C$. Then 
$C$ admits a $g^1_d$
 with $d\in \lbrace  4, 5, 6 \rbrace$.
\end{proposition}
\begin{proof}
 If for some $D\in V^{3}_{4}(L)$ one had  $h^0(D)=2$ then $C$ is $4$-gonal,
  while if for each $D\in V^{3}_{4}(L)$ we had
$h^0(D)=1$ then $h^0(D_{1}+D_{2})$ would belong $ \lbrace 2, 3, 4 \rbrace$ for $D_{1}$ and $D_{2}$ in $ V^{3}_{4}(L)$.
Therefore three cases can occur.

 If for general $D_{1}$ and $D_{2}$ in $V^{3}_{4}(L)$ we have $ h^0(D_{1}+D_{2})=2$ then arguing as in the proof of \cite[Theorem 4.6]{A. B}, we obtain a map
$\phi: C\rightarrow \mathbf{P}^3$ such that 
\[(\deg \phi)(\deg \phi (C)-1)=8.\]
According to this equality; if $\deg \phi=1$, which is the case that $\phi(C)$, as well as $C$, is a space curve of degree $9$; then projecting from a point of $C$ into $\mathbf{P}^2$ we obtain a singular plane curve of degree $8$.
Such a curve has to admit a $g^1_6$.

If we had $\deg \phi=2$, which is the same as $\phi(C)$ to be a space quintic, then $\phi(C)$ would admit a $g^1_2$ and therefore $C$ admits a $g^{1}_{4}$.

 Lastly 
 $\deg \phi=4$ and $\phi(C)$ is a space cubic curve which has to be a rational normal curve. Therefore $C$ is a $4$-sheeted covering of a rational 
normal space curve, which means that $C$ admits a $g^{1}_{4}$.

 Assume that we are in the second case; i.e. for general $D_{1}, D_{2}\in V^{3}_{4}(L)$ we have $ h^0(D_{1}+D_{2})=3$, by which we conclude that
$\dim C^{2}_{8} \geq 2$. This by removing a general point $p\in C$ from the divisors in $C^{2}_{8}$, implies that $\dim C^{1}_{7} \geq 2$. According to the results of \cite{A-S2}, we obtain that $C^{1}_{6}$ is non-empty. Therefore $C$ has to admit a $g^{1}_{6}$.

 Finally we assume that for general $D_{1}, D_{2}\in V^{3}_{4}(L)$ one has $ h^0(D_{1}+D_{2})=4$ and we obtain $\dim C^{3}_{8} \geq 2$. As in the previous, we find that $C$ has to admit a $g^{1}_{5}$.
\end{proof}
\begin{lemma}\label{lemma2}
For $p\in C$ and a very ample line bundle $L$ on $C$ we have
$$\dim V^{d-1}_{d}(L)\leq \dim V^{d-1}_{d}(L(-p))\leq \dim V^{d-1}_{d}(L) +1.$$
In particular; $V^{d-1}_{d}(L)$ is non-empty provided that $\dim V^{d-1}_{d}(L(-p))\geq 1$ for some $p\in C$.
\end{lemma}
\begin{proof}
The first inequality is immediate from $V^{d-1}_{d}(L) \subseteq V^{d-1}_{d}(L(-p))$.\\
To establish the second inequality, we use \cite[Theorem 4.1(6)]{A-S2} with $k=d-1$ and we obtain:
\begin{equation}
\dim V^{d-1}_{d}(L)+2\geq \dim V^{d-1}_{d+1}(L).	 \label{1}
\end{equation}
Consider moreover that for $p\in C$ and for any $D\in V^{d-1}_{d}(L(-p))$ one has $D+p\in V^{d-1}_{d+1}(L)$. This proves
$\dim V^{d-1}_{d}(L(-p))\leq \dim V^{d-1}_{d+1}(L)-1.$
Comparing with (\ref{1}) we obtain the result.
%$$\dim V^{1}_{d}(L)+1\geq \dim V^{1}_{d}(L(-p)).$$
\end{proof}
\begin{theorem}\label{theorem1}
Assume that $C$ is a smooth projective non-hyper elliptic curve of genus $g$ with $g\geq 9$.
If for some very ample line bundle $L$ on $C$ there exist integers $r$, $d$ with $0< 2r < d\leq h^0(L)+r-4$ such that
$\dim V^{d-r}_{d}(L)=d-r-2$, then $C$ admits a $g^{1}_{d}$ with $d\in \lbrace 3, 4, 5, 6 \rbrace$.
\end{theorem}
\begin{proof}
We assume that $V^{d-r}_{d}(L)$ is irreducible.
Removing $q$ from the series in $V^{d-r}_{d}(L)$ we obtain a $(1+(d-r-2)-1)$-dimensional family of divisors $\bar{D}$ belonging to $ V^{d-r}_{d-1}(L)$, so we would have $ \dim V^{d-r}_{d-1}(L)\geq d-r-2$. This together with \cite[Theorem 4.2]{A. B}, implies that either
$\dim V^{d-r}_{d-1}(L)=d-r-2$ or $\dim V^{d-r}_{d-1}(L)=d-r-1$. Theorem 4.6 of \cite{A. B} forces $C$ to be exceptional in the latter case.
 Iterating this process we find that either $\dim V^{d-1}_{d}(L)=d-3$ or $\dim V^{d-1}_{d}(L)=d-2$, for some $d\leq h^0(L)-3$. The latter case forcing $C$ to be exceptional, we proceed in the first case.
 
Assume that $L$ is a very ample line bundle with minimum $h^{0}(L)$ among those very ample line bundles $H$,  for which $V^{d-1}_{d}(H)$ is of dimension $d-3$.

Under this minimality assumption on $h^{0}(L)$, two cases can occur.  For a general $p\in C$ the line bundle $L(-p)$ fails to be very ample, where we wolud have  $\dim V^{2}_{3}(L)=1$ forcing $C$
to be exceptional by Lemma \ref{lemma1}.

The second possiblity is that $h^{0}(L)=d+3$. This case consists of three subcases.
If for general $p\in C$ the line bundle $L(-p)$ fails to be very ample then we are reduced to the previous case.

The subcase $d=h^{0}(L)-3=3$ implies that
$\dim V^{2}_{3}(L)=0$, which together with \cite[Theorem 4.2]{A. B} implies that either $\dim V^{3}_{4}(L)=1$ or $\dim V^{3}_{4}(L)=2$. Using proposition \ref{proposition} the equality  $\dim V^{3}_{4}(L)=1$ implies the assertion.
 While \cite[theorem 4.6]{A. B}, forces $C$ to be exceptional in the case $\dim V^{3}_{4}(L)=2$.

  As the last case; assume that for general $p\in C$ the line bundle $L(-p)$ is very ample with $d\geq 4$
 and $\dim V^{d-1}_{d}(L)=d-3$, where $d=h^0(L)-3$.
  Having discussed the case $d=4$ in Proposition \ref{proposition}, we assume that $d\geq 5$. 
 If $\dim V^{d-2}_{d-1}(L)=d-3$, then $C$ would be exceptional. Assuming $\dim V^{d-2}_{d-1}(L)\neq d-3$ we make a claim
 
\qquad \qquad \qquad  \qquad  \qquad  \qquad  \textit{Claim:} $\dim V^{d-1}_{d}(L(-p))\neq d-3.$   

Having proved the claim; we 
use lemma \ref{lemma2} together with \cite[Theorem 4.2]{A. B} to obtain
$$\dim V^{d-1}_{d}((L(-p))=d-2.$$
 This by \cite[Theorem 4.6]{A. B}
forces $C$ to be exceptional.
 To end the proof; we notify that the bi-eliptic case is excluded via Example \ref{example}.\\
\end{proof}
\textit{Proof of the Claim:}
Equivalent to the claim we prove that; if $X$ is an irreducible component of $V^{d-1}_{d}(L(-p))$ such that $\dim X=\dim V^{d-1}_{d}(L(-p))$, then a general member of an irreducible component $V$ of $V^{d-1}_{d}(L)$ fails to be a general member of $X$.
To do this; consider that for each $D\in V$
   there exists $p\in C$ such that $(p+C_{d-1})\cap V^{d-1}_{d}(L) \neq \varnothing$.
    Since $\dim V\geq 2$, moving $D$ in $V$ this $p$ has to move in an open subset of $C$. Otherwise removing $p$ from the divisors in $V$ we obtain 
 $\dim V^{d-2}_{d-1}(L)=d-3$, which we had assumed won't occur. 
 
If a general divisor $D\in V$ turns to be a general member of $X$, then for general $p, q\in C$, divisors of type $D-p+q$ belonging to $X$ lie on $V$, which is absurd by genericity of $q$ and $D$. 
%Summarizing we obtain $\dim V^{d-1}_{d}(\mathcal{H}(-p))\geq d-\gamma$.
This implies that $d-3=\dim V^{d-1}_{d}(L) < \dim V^{d-1}_{d}(L(-p))$. 

\begin{remark}\label{remark1}
(a)  Lemma \ref{lemma1} shows that, unlike the variety of special divisors, the equality $\dim V^{2}_{3}(L)=1$ for some very ample line bundle $L$ on $C$, doesn't imply  $3$-gonality of $C$.

(b) Based on the proof of Theorem \ref{theorem1}, we know the shape of a general element in a specific irreducible component of $V^{d-1}_{d}(L)$, when $L$ is 
a very ample line bundle on a bi-elliptic curve. This obvious generalization from the canonical case to the case of secant loci of very ample line bundles, remains no longer true when $C$ is $3$-gonal or $4$-gonal. For example; an easy calculation clarifies that the unique $g^{1}_{3}$ on a $3$-gonal curve, as well as a $g^{1}_{4}$ on a $4$-gonal curve, does not belong to $V^{2}_{3}(2K_{C})$, $V^{3}_{4}(2K_{C})$ respectively. Unfortunately, we don't have any knowledge about the shape of a general member of an element of $V^{d-1}_{d}(L)$ in the case that $C$ is $3$-gonal or $4$-gonal.

(c) Each class of curves appeared in theorem \ref{theorem1}, has a member admitting a very ample line bundle $L$ such that $\dim V^{r}_{d}(L)=d-r-2$ for some integers $r, d$ with $0<2r<d\leq h^0(L)+r-4$.

In fact for a $3$-gonal curve of genus $g$ with $d\leq g-2$, we have
$\dim C^2_{d}=\dim V^{d-2}_{d}(K)=d-4.$
See \cite[page 198]{ACGH}.
For a $4$-gonal curve we have
$\dim C^1_{d}=\dim V^{d-1}_{d}(K)=d-3$
with $4\leq d\leq g-2$. 

On a $5$-gonal curve with $p\in C$ as a base point of $K(-g^1_5)$, setting $L=K(-p)$, we observe that  for general points $q_1, q_2, \cdots, q_t$ on $C$;
divisors of type $D=g^1_5+q_1+q_2+\cdots +q_t$, lie on $ V^{t+3}_{t+5}(L)$. This implies that $\dim  V^{t+3}_{t+5}(L)=t+1$.

As in the previous case if $C$ is a $6$-gonal curve,
then
% with $p\in C$ as a base point of $K(-g^1_5)$ setting $L=K(-p)$, we observe that  for general points $q_1, q_2, \cdots, q_t$ on $C$,
divisors of type $D=g^1_6+q_1+q_2+\cdots +q_t$, lie on $ V^{t+4}_{t+6}(L)$. This implies that 
$\dim  V^{t+4}_{t+6}(L)\geq t+1$. If $\dim  V^{t+4}_{t+6}(L)=t+1$, then $V^{t+4}_{t+6}(L(-q))$ would be $t+2$ dimensional for some $q\in C$.

\end{remark}

\section{An Improved Argument}
In this section we affirmatively answer problem \ref{problem}, for secant loci of very ample line bundles. At the same time, removing restrictions on the genus of $C$ imposed by M. Coppens in the spacial case $L=K_{C}$, we extend it considerably.

 Having treated the case $\gamma =0$ in \cite[Theorem 4.2]{A. B}, we will assume that 
$\gamma\geq 1$. 

\begin{lemma}\label{lemma3}
Assume that $H$ is a very ample sub-line bundle of the very ample line bundle $L$ such that $\dim V^{d-1}_{d}(L)= \dim V^{d-1}_{d}(H)$, for some integer $d$ with $d\leq h^{0}(L)-1$. If $V^{d-2}_{d-1}(L)$ is non empty and $\dim V^{d-2}_{d-1}(H)=\dim V^{d-2}_{d}(H)-1$, then
$\dim V^{d-2}_{d-1}(L)= \dim V^{d-2}_{d-1}(H).$
\end{lemma}
\begin{proof}
 Assume that $X$ is an irreducible component of $ V^{d-1}_{d}(H)$ in common with $ V^{d-1}_{d}(L)$ such that 
 $\dim X=\dim V^{d-1}_{d}(H)$. If for some $p\in C$ one had $(p+C_{d-1})\cap X=X$, then 
 we obtain $\dim V^{d-2}_{d-1}(L)= \dim V^{d-1}_{d}(H)$, which is impossible.
 Therefore from the equality 
 $$X=\cup_{p\in C}[(p+C_{d-1})\cap X]$$ we conclude that for general $q\in C$, the closed subscheme $Y:=(q+C_{d-1})\cap X$ is of codimension $1$ in $X$. Consider now that, using genericity of $q$,  divisors  $D\in C_{d-1}$ such that $q+D\in X$, belong to  $V^{d-2}_{d-1}(L)$ and $V^{d-2}_{d-1}(H)$. This implies the assertion.\\

\end{proof}
\begin{theorem}\label{theorem2}
Let $C$ be a non-hyper elliptic curve of genus $g\geq 9$. Assume that for some very ample line bundle $L$ on $C$ and integers $d, \gamma$ with 
$d\geq 3, \gamma \geq 1$ such that
$h^0(L)\geq 2\gamma +4$,
 $\gamma+3 \leq d \leq h^0(L)-1-\gamma$; 
 one has $\dim V^{d-1}_{d}(L)=d-1-\gamma$. Then $V^{\gamma+2}_{\gamma+3}(L)$ has to be $2$-dimensional.
\end{theorem}
\begin{proof} 
We use induction on $\gamma$. 
Assume $ \gamma=1$. 
If $L(-p)$ fails to be very ample, then $\dim V^{2}_{3}(L)=1$. The equality $\dim V^{3}_{4}(L)=3$
  leads to $\dim V^{d-1}_{d}(L)\geq d-1$, which is absurd. Therefore $\dim V^{3}_{4}(L)=2$ and we get the result.
 Assume $d\geq 4$ and 
 let $\mathcal{H}$ be a very ample sub-line bundle of $L$ with minimum $h^0(\mathcal{H})$ among those very ample sub-line bundles $\Gamma $ of $L$ such that $\dim V^{d-1}_{d}(\Gamma)=d-2$.
 
For general $p\in C$ if the line bundle $\mathcal{H}(-p)$ fails to be very ample, then as in the previous case we obtain 
$\dim V^{3}_{4}(\mathcal{H})=2$. Observing $\dim V^{d-1}_{d}(\mathcal{H})=d-2$ we obtain $\dim V^{\beta +2}_{\beta +3}(\mathcal{H})=\beta+1$  for $0\leq \beta \leq d-3 $.
Lemma \ref{lemma3} applied to $\mathcal{H}$ and $L$ gives $\dim V^{3}_{4}(L)=2$.

 If for general $p\in C$, the line bundle $\mathcal{H}(-p)$ turns to be very ample, then
  $h^{0}(\mathcal{H})=d+2$ and $\dim V^{d-1}_{d}(\mathcal{H})=\dim V^{d-1}_{d}(\mathcal{H}(-p))$. But an argument as in the proof of our claim in theorem \ref{theorem1}, excludes this possibility. This completes the case $\gamma =1$.

Assume $\gamma \geq 2$. We prove that there exists a very ample sub-line bundle $\Gamma$ of $L$ such that $\dim V^{\gamma +2}_{\gamma + 3}(\Gamma)=2$.
 
Let $\mathcal{H}$ be choosen as in the case $\gamma=1$ with  $\dim V^{d-1}_{d}(\mathcal{H})=d-1-\gamma$.
If for general $p\in C$ the line bundle $\mathcal{H}(-p)$ fails to be very ample, then we obtain
$\dim V^{2}_{3}(\mathcal{H})=1$. Implying  
  $$\dim V^{d-1}_{d}(\mathcal{H})\geq d-2>d-1-\gamma=\dim V^{d-1}_{d}(\mathcal{H}),$$
  this case would be absurd. Meanwhile; the possibility 
 $h^{0}(\mathcal{H})=d+1+\gamma$
with failing very ampleness of $\mathcal{H}(-p)$ for general $p\in C$, is excluded similarly.

  Assume that for general $p\in C$ the line bundle $\mathcal{H}(-p)$ is very ample with 
 $\dim V^{d-1}_{d}(\mathcal{H})=d-1-\gamma$ and $d=h^0(\mathcal{H})-1-\gamma \geq 4$. We distinguish two main cases.
 
 Assuming
 $\dim V^{d-2}_{d-1}(\mathcal{H}(-p))=d-1-\gamma$ 
 as the first main case,
 the induction hypothesis implies that
 $\dim V^{\gamma +1}_{\gamma +2}(\mathcal{H}(-p))=2$.
 Applying \cite[Theorem 4.1]{A-S2},
$V^{\gamma +2}_{\gamma +3}(\mathcal{H}(-p))$ would be of dimension $3$ or $4$. 
The latter case implies $\dim V^{d-1}_{d}(\mathcal{H}(-p))\geq 4+(d-\gamma-3)=d-\gamma+1$ which using $\dim V^{d-1}_{d}(\mathcal{H})=d-1-\gamma$ contradicts lemma \ref{lemma2}. Therefore $\dim V^{\gamma +2}_{\gamma +3}(\mathcal{H}(-p))=3$.
Using lemma
\ref{lemma2}, $V^{\gamma +2}_{\gamma+3}(\mathcal{H})$ would be $2$ or $3$ dimensional. Three dimensionality of $ V^{\gamma +2}_{\gamma+3}(\mathcal{H})$ implies that $\dim V^{d-1}_{d}(\mathcal{H})\geq d-\gamma$, which is absurd.
Therefore $\dim V^{\gamma +2}_{\gamma + 3}(\mathcal{H})=2$.

As the second main case we assume 
 $\dim V^{d-2}_{d-1}(\mathcal{H}(-p))\neq d-1-\gamma$ and we
 claim;
 $$\dim V^{d-1}_{d}(\mathcal{H}(-p))\geq d-\gamma.$$
 
 The claim can be proved as in the proof of our claim in theorem \ref{theorem1}, so we omit its proof.
 We use Lemma \ref{lemma2} and obtain 
$\dim V^{d-1}_{d}(\mathcal{H}(-p))=d-1-(\gamma -1).$
This again by induction hypothesis asserts;
$\dim V^{\gamma +1}_{\gamma +2}(\mathcal{H}(-p))=2$, by which
 as in the first main case we obtain $\dim V^{\gamma +2}_{\gamma + 3}(\mathcal{H})=2$.
To end the proof, we apply Lemma \ref{lemma3} and obtain: 
$$\dim V^{\gamma +2}_{\gamma + 3}(L)=\dim V^{\gamma +2}_{\gamma + 3}(\mathcal{H})=2.$$
 \end{proof}
 Through theorem \ref{theorem3}, we give an application to our results, specifically theorem \ref{theorem2} and lemma \ref{lemma2}.
%  Finally we report a dimension computing in Brill-Noether theory, as an application of our results:
 \begin{theorem}\label{theorem3}
 Assume that $C$ is a non-hyper elliptic curve of genus $g\geq 9$ and $L$ is a very ample line bundle on $C$ such that $V^{h^0(L)-d}_{h^0(L)-d+1}(L)=\emptyset$ for some integer $d$ with $[\frac{h^0(L)+3}{2}]\leq d\leq h^0(L)-1$.
 Then 
$ \dim V^{d-1}_{d}(L)=2d-h^0(L)-1$.

 In particular; if $C$ is a general curve of genus $g$, then for integers $d\in \lbrace  [\frac{g+3}{2}], \cdots , g-1\rbrace$, the varieties $C^1_d$ and $W^1_d$ are of expected dimensions $2d-g-1$, $2d-g-2$, respectively.
\end{theorem}
\begin{proof}
We set $d=h^0(L)-k$ and use induction on $k$. Consider that the case $k=1$ is immediate by lemma \cite[Lemma 2.1]{A-S}. 
By our emptiness assumption we obtain that $V^{2}_{3}(L)=\emptyset$. Therefore 
for general $p\in C$, the line bundle $L(-p)$ is very ample. In addition,
$V^{h^0(L)-d-1}_{h^0(L)-d}(L(-p))$ contained in
$V^{h^0(L)-d}_{h^0(L)-d+1}(L)$ would be empty.
Using the induction hypothesis we obtain $\dim V^{d-1}_{d}(L(-p))=2d-h^0(L)$. On use of emptiness of $V^{h^0(L)-d}_{h^0(L)-d+1}(L)$, 
 we find that if $X$ is an irreducible component of $V^{d-1}_{d}(L(-p))$ such that $\dim X=\dim V^{d-1}_{d}(L(-p))$, then a general member of an irreducible component, $V$, of $V^{d-1}_{d}(L)$ fails to be a general member of $X$. This, by Lemma \ref{lemma2} would give the result.

A general curve of genus $g$ has gonality equal to $[\frac{g+3}{2}]$. Threfore the emptiness assumption is immediate for $L=K_{C}$ on general curves. 

\end{proof}
\begin{remark}\label{remark2}
(a) Although in the case $\gamma =1$ of theorem \ref{theorem2} using \cite[Theorem 4.6]{A. B} we find $C$ an exceptional curve 
 but;
 based on remark \ref{remark1}(b), this is useless to conclude theorem \ref{theorem2}  when $\gamma =1$. So we were forced to make extensions in proof of theorem \ref{theorem1} to obtain the result directly in the case  $\gamma =1$.

(b) Notice that theorem \ref{theorem1} is not a consequence of theorem \ref{theorem2} for specific values of $\gamma$, e.g. for $\gamma=1$
 or $\gamma=2$.  In fact; in the case $\gamma =1$, based on theorem \ref{theorem2}, the equality $\dim V^{d-1}_{d}(L)=d-2$ leads to $\dim V^{3}_{4}(L)=2$. This by \cite[Theorem 4.1]{A-S2} implies that either $\dim V^{2}_{3}(L)=0$ or $\dim V^{2}_{3}(L)=1$. Although lemma \ref{lemma1} implies exceptionality of $C$ in the latter case, but  observing the uncontrollable nature of divisors in $V^{d-r}_{d}(L)$, one can not conclude $3$-gonality, $4$-gonality or bi-ellipticity of $C$ in the first case. The case $\gamma=2$ is obviously more complicated.
  
(c) G. Farkas gives in \cite{F}, numerical conditions that they ensure 
emptiness of $V^{h^0(L)-d}_{h^0(L)-d+1}(L)$ for some inregers $d$ and various line bundles on general curves. Therefore theorem \ref{theorem3} would be applicable for line bundles and integers having these conditions. 
Meanwhile there are cases that theorem \ref{theorem3} can be applied without using Farkas' results. For example for
very ample line bundles on 
non-exceptional curves of genus $g\geq 9$ we find $\dim V^{h^0(L)-3}_{h^0(L)-2}(L)=d-3$. Additionally for
very ample line bundles on general curves of genus $g\geq 11$, according to the gonality of general curves and using proposition \ref{proposition} we find $\dim V^{h^0(L)-4}_{h^0(L)-3}(L)=d-4$.
\end{remark}

\end{document}